\documentclass[12pt]{amsart}
\usepackage{epsfig}
\usepackage{graphics}
\usepackage{color}
\usepackage{dcpic, pictexwd}
\usepackage{tikz}
\usepackage[utf8]{inputenc}
\usepackage{amsmath}
\usepackage{amsfonts}
\usepackage{amssymb}
\usepackage{calligra}
\usepackage{calrsfs}
\usepackage[mathscr]{euscript}
\usepackage{pinlabel}

\newtheorem{theorem}{Theorem}

\newtheorem{lemma}[theorem]{Lemma}

\newtheorem{corollary}[theorem]{Corollary}

\theoremstyle{definition}

\theoremstyle{remark}

 \def\R{{\mathbb{R}}}
 \def\Z{{\mathbb{Z}}}

 \def\S{{\Sigma}}

\def\mod{{\rm Mod}}

 \def\SL{{\rm SL}}
 \def\Sp{{\rm Sp}}

 \def\G{{\mathcal{G}}}
\def\S{{\mathcal{S}}}

\begin{document}

\newenvironment{prooff}{\medskip \par \noindent {\it Proof}\ }{\hfill
$\square$ \medskip \par}
    \def\sqr#1#2{{\vcenter{\hrule height.#2pt
        \hbox{\vrule width.#2pt height#1pt \kern#1pt
            \vrule width.#2pt}\hrule height.#2pt}}}
    \def\square{\mathchoice\sqr67\sqr67\sqr{2.1}6\sqr{1.5}6}
\def\pf#1{\medskip \par \noindent {\it #1.}\ }
\def\endpf{\hfill $\square$ \medskip \par}
\def\demo#1{\medskip \par \noindent {\it #1.}\ }
\def\enddemo{\medskip \par}
\def\qed{~\hfill$\square$}

 \title[Mapping class group is generated by three involutions]
{Mapping class group is generated by three involutions}

\author[Mustafa Korkmaz]{Mustafa Korkmaz}
\date{\today}

 \address{Department of Mathematics, Middle East Technical University, 06800
  Ankara, Turkey
  }
 \address{Department of Mathematics and Statistics, UMASS-Amherst, Amherst MA, 01003
  USA
  }
 \email{korkmaz@metu.edu.tr}
\begin{abstract}
We prove that the mapping class group of a closed connected orientable surface of genus at least eight
is generated by three involutions.
\end{abstract} 
 \maketitle

\section{Introduction}
The mapping class group $\mod(\Sigma_g)$ of a closed connected orientable 
surface $\Sigma_g$ is defined as the group of isotopy classes of 
orientation--preserving self--diffeomorphisms of $\Sigma_g$. We are interested in
generating $\mod(\Sigma_g)$ by the least number of involutions. 

The group $\mod(\Sigma_g)$ cannot be generated by two 
involutions, because, for example, it contains nonabelian free groups. Thus any generating set 
consisting of involutions must contain at least three elements.
The purpose of this paper is to prove that the mapping class
group can be generated generated by three involutions 
if the genus is at least eight, answering a question in~\cite{BrendleFarb}.

\begin{theorem} \label{thm:1}
The mapping class group $\mod(\Sigma_g)$ is generated by 
three involutions if $g\geq 8$, and by four involutions if $g \geq 3$. 
\end{theorem}

After the works of Luo and Brendle-Farb as explained below, 
Kassabov~\cite{Kassabov}  obtained a generating set consisting of
four involutions if $g\geq 7$. He also proved results for lower genus
mapping class groups.

For homological reasons, the groups $\mod(\Sigma_1)$ and $\mod(\Sigma_2)$ 
cannot be generated by involutions. Since there is a surjective 
homomorphism from $\mod(\Sigma_g)$ onto the symplectic group 
$\Sp (2g,\Z)$, 
the latter group  is also generated by three involutions for $g\geq 8$. 

\begin{corollary} 
The symplectic group $\Sp(2g,\Z)$ is generated by 
three involutions if $g\geq 8$, and by four involutions if $g\geq 3$.
\end{corollary}

Here is a brief history of the problem of finding the generating sets for $\mod(\Sigma_g)$
with various properties. Dehn~\cite{Dehn}
obtained a generating set for $\mod(\Sigma_g)$ consisting of $2g(g-1)$ Dehn twists. About 
a quarter century later, Lickorish~\cite{Lickorish} 
showed that it can be generated by $3g-1$ Dehn twists, and Humphries~\cite{Humphries} 
reduced this number to $2g+1$. He showed, moreover, that $2g+1$ is minimal: the cardinality of
every generating set of $\mod(\Sigma_g)$ consisting of Dehn twists is at least $2g+1$.  

If one does not require generators to be Dehn twists, by using the generating set 
obtained by Lickorish, it is easy to find a generating set consisting of 
an element of order $g$ and three Dehn twists. It is also possible to get a generating set with three
elements (see Corollary~\ref{cor:gen1} below).
Lu~\cite{Lu} obtained a generating set consisting of three elements, two of which are of finite order. 
A minimal generating set was first  obtained by Wajnryb~\cite{Wajnryb1996} who proved that
$\mod(\Sigma_g)$ can be generated by two elements; one is of order $4g+2$ and the other
is a product of a right and a left Dehn twist about disjoint curves. 
In~\cite{Korkmaz}, we showed that $\mod(\Sigma_g)$ 
is generated by an element or order $4g+2$ and a Dehn twist, improving Wajnryb's result.
 
The interest in the problem of finding generating sets for the mapping class group 
consisting of finite order elements goes back to 1971:
Maclachlan~\cite{Maclachlan} proved that $\mod(\Sigma_g)$ is generated by the conjugates of
two elements of orders $2g+2$ and $4g+2$. He deduced from this
that the moduli space is simply-connected. McCarthy and Papadopoulos~\cite{McCarthyPapadopoulos} showed that 
for $g\geq 3$ the group $\mod(\Sigma_g)$ is generated normally by a single involution; 
i.e., it is generated by an involution and its conjugates. Luo~\cite{Luo}
observed that every Dehn twist is a product of six involutions for $g\geq 3$. It then follows from the fact that
the mapping class group is generated by $2g+1$ Dehn twists,  the group $\mod(\Sigma_g)$ can be
generated by $12g+6$ involutions. Luo also asked whether it is possible to generate the mapping class group
by torsion elements, where the number of generators
is independent of the genus (and boundary components in the case
the surface has boundary).

Brendle and Farb~\cite{BrendleFarb} answered Luo's question affirmatively 
 by proving that $6$ involutions, and also $3$ torsion elements,  
 generate the mapping class group for $g\geq 3$. 
Kassabov~\cite{Kassabov} improved this result further, proving that $\mod(\Sigma_g)$ is generated by 
$4$ (resp. $5$ and $6$) involutions if $g\geq 7$ (resp. $g\geq 5$ and $g\geq 3$). 
 We proved  in~\cite{Korkmaz} 
that the minimal number of torsion generators of 
$\mod(\Sigma_g)$ is $2$, by showing that 
the mapping class group $\mod(\Sigma_g)$ can be generated by 
two elements of order $4g+2$. See also~\cite{Monden},~\cite{Du2} and~\cite{Lanier} 
for generating  
the mapping class group by torsion elements of various orders.

In order to prove the main results of their papers, Brendle-Farb and also Kassabov write
a Dehn twist as a product of four involutions, and use a generating set consisting
of a torsion element and three Dehn twists. It is easier to write a 
product of two opposite Dehn twists about disjoint curves as a product of two involutions. 
Therefore, Wajnryb's generating 
set, consisting of an element of order $4g+2$ and a product of two opposite Dehn twists,
looks like a good candidate to use in order to find a small number of involution generators. 
However, the element of order $4g+2$ cannot be
written as a product of two (orientation--preserving) involutions, because 
it follows from ~\cite{BCGG} that the group $\mod(\Sigma_g)$ 
does not contain a dihedral subgroup of order $8g+4$.
In order to implement this idea, we find new generating sets for the mapping
class group. 
 
\medskip

\noindent
{\bf Acknowledgments.}
I thank UMass-Amherst 
for its generous support and wonderful research environment,
where this research was completed while I was visiting on  leave from Middle East 
Technical University in the academic year 2018--2019.  I thank \.Inan\c{c} Baykur for 
his interest and comments on the paper. 
\bigskip

\section{Background and results on mapping class groups}

In this article we will only consider closed surfaces.
Let $\Sigma_g$ be a closed connected oriented surface of genus $g$ embedded 
in $\R^3$, as illustrated in Figure~\ref{fig1}, in such a way 
that it is invariant under the two rotations 
$\rho_1$ and $\rho_2$; $\rho_1$ is the rotation by $\pi$ about the $z$--axis
and $\rho_2$ is the rotation by $\pi$ about the line $z=-\tan(\pi/g) y$, $x=0$.
The mapping class group $\mod(\Sigma_g)$ of $\Sigma_g$ 
is the group of isotopy classes of orientation--preserving self--diffeomorphisms of $\Sigma_g$. 

Throughout the paper diffeomorphisms are considered up to isotopy. Likewise, curves are considered up
to isotopy. Simple closed curves are denoted by the lowercase letters $a,b,c$, with indices,
while the right Dehn twists about them by the corresponding capital letters $A,B,C$. 
Occasionally, we will write $t_a$ for the right Dehn twist about $a$. 
For the composition of diffeomorphisms, we use the functional notation: 
$fh$ means that $h$ is applied first. The notations $a_i,b_i,c_i$ and $d_1,d_2$
will always  denote the curves shown in Figure~\ref{fig1}.

Let us review the basic relations among Dehn twists that we need below.  
For the proofs the reader is referred to~\cite{FarbMargalit}.
First of all, if $f: \Sigma_g\to \Sigma_g$ is a diffeomorphism
and if $a$ is a simple closed curve on $\Sigma_g$, then $ft_af^{-1}=t_{f(a)}$.

\medskip \noindent
\textbf{Commutativity:}  If $a$ and $b$ are two disjoint simple closed curves on $\Sigma_g$, then $AB=BA$.

\medskip \noindent
\textbf{Lantern relation:}  This relation 
 was discovered by Dehn~\cite{Dehn} in 1930s, and rediscovered and popularized by 
Johnson~\cite{Johnson} in
1979. Suppose that $x_1,x_2,x_3,x_4$ are pairwise disjoint simple closed curves on $\Sigma_g$
bounding a sphere $S$  with  four boundary components. Let us 
choose a point $P_i$ on $x_i$ for each $i=1,2,3,4$. 
For $j=1,2,3$, choose a properly embedded arc $\gamma_j$ on $S$ 
connecting $P_4$ and $P_j$, so that they are disjoint in the interior of $S$. Suppose that 
in a small neighborhood of the point $P_4$, the arcs are read as $\gamma_1,\gamma_2,\gamma_3$  in 
the clockwise order. Let $y_j$ be the boundary component of a regular neighborhood of $x_4\cup\gamma_j\cup x_j$
lying on $S$. Then the Dehn twists about the seven curves 
$x_1,x_2,x_3,x_4,y_1,y_2,y_3$ satisfy the lantern relation 
\[
X_1X_2X_3X_4=Y_1Y_2Y_3.
\]

 It was proved by Dehn~\cite{Dehn} that the mapping class group $\mod(\Sigma_g)$ 
 is generated by Dehn twists about finitely many nonseparating simple closed curves.  
 Lickorish also obtained the same result and showed in~\cite{Lickorish} that it is generated by 
 the Dehn twists 
 \[
 A_1,A_2,\ldots ,A_g, B_1,B_2,\ldots ,B_g,C_1,C_2,\ldots ,C_{g-1}
 \] 
 about the curves shown in Figure~\ref{fig1}. We state this
 fact as a theorem by adding one more Dehn twist for the sake of symmetry.

\begin{theorem} {\rm\bf(Dehn-Lickorish)}\label{thm:Dehn-Lick}
The mapping class group $\mod(\Sigma_g)$ is generated by the set
$\{ A_1,A_2,\ldots ,A_g, B_1,B_2,\ldots ,B_g,C_1,C_2,\ldots ,C_g\}.$
\end{theorem}

As was shown by Humphries~\cite{Humphries}, the Dehn twists $A_3,A_4,\ldots,A_g,C_g$ 
can be removed from the above generating set, leaving $2g+1$ Dehn twists. He also proved
that it is minimal, in the sense that $2g$ or less Dehn twists cannot generate $\mod(\Sigma_g)$
if $g\geq 2$.

As a corollary to Dehn-Lickorish Theorem, it is easy to show that $\mod(\Sigma_g)$
can be generated by four elements, and also by three elements. We state it as the next corollary. 
To this end, let 
$R$ be the rotation by $2\pi/g$ about the $x$--axis represented in Figure~\ref{fig1},
so that $R=\rho_1\rho_2$. Thus, \mbox{$R(a_k)=a_{k+1}$,} $R(b_k)=b_{k+1}$ and
$R(c_k)=c_{k+1}$.

\begin{corollary} \label{cor:Dehn-Lick}
The mapping class group $\mod(\Sigma_g)$ is generated 
\begin{itemize} 
	\item[{\rm (i)}] by the four elements $R, A_1,B_1,C_1$, and also 
	\item[{\rm (ii)}] by the three elements $ R, A_1, A_1B_1C_1$.
\end{itemize}
\end{corollary}
\begin{proof}
 The claim (i) follows from Theorem~\ref{thm:Dehn-Lick} and the fact that 
 $R^{k}$ conjugates the Dehn twists $A_1,B_1$ and $C_1$ to  
  $A_{k+1}, B_{k+1}$ and $C_{k+1}$, respectively.
 The claim (ii) then follows from (i) together with the fact that $A_1B_1C_1(a_1)=b_1$
 and $A_1B_1C_1(b_1)=c_1$. 
\end{proof}

Although we will not directly use the above corollary, we stated it anyway, because
a version of it will be proved in the next section and will be used in 
our proof of Theorem~\ref{thm:1}. 

\begin{figure}
\begin{tikzpicture}[scale=0.45]
\begin{scope} [xshift=0cm]
 \draw[very thick, violet] (0,0) circle [radius=5.5cm];
 \draw[very thick, violet] (0,2.7) circle [radius=0.8cm]; 
 \draw[very thick, violet, rotate=80] (0,2.7) circle [radius=0.8cm]; 
 \draw[very thick, violet, rotate=-80] (0,2.7) circle [radius=0.8cm]; 
 \draw[very thick, violet, fill ] (0,-2.7) circle [radius=0.03cm]; 
 \draw[very thick, violet, fill, rotate=15] (0,-2.7) circle [radius=0.03cm]; 
\draw[very thick, violet, fill, rotate=-15] (0,-2.7) circle [radius=0.03cm]; 
\draw[thick, green,  rounded corners=10pt] (-0.05, 3.5) ..controls (-0.6,3.8) and (-0.6,5.2).. (-0.05,5.5) ;
\draw[thick, green, dashed, rounded corners=10pt] (0.05,3.5)..controls (0.6,3.8) and (0.6,5.2).. (0.05,5.5) ;
\node at (-1,4.7) {$a_2$};
\draw[thick, green, rotate=80, rounded corners=10pt] 
	(-0.05, 3.5) ..controls (-0.6,3.8) and (-0.6,5.2).. (-0.05,5.5) ;
\draw[thick, green, dashed, rotate=80, rounded corners=10pt] 
	(0.05,3.5)..controls (0.6,3.8) and (0.6,5.2).. (0.05,5.5) ;
\node at (-4.6,-0.1) {$a_1$};
\draw[thick, green, rotate=-80, rounded corners=10pt] 
	(-0.05, 3.5) ..controls (-0.6,3.8) and (-0.6,5.2).. (-0.05,5.5) ;
\draw[thick, green, dashed, rotate=-80, rounded corners=10pt] 
	(0.05,3.5)..controls (0.6,3.8) and (0.6,5.2).. (0.05,5.5) ;
\node at (4.7,-0.1) {$a_3$};
 \draw[thick, blue] (0,2.7) circle [radius=1.1cm]; 
 \draw[thick, rotate=80, blue] (0,2.7) circle [radius=1.1cm]; 
 \draw[thick, blue, rotate=-80] (0,2.7) circle [radius=1.1cm]; 
\node at (-3.5,2) {$b_1$};
\node at (3.2,2) {$b_3$};
\node at (1.4,3.7) {$b_2$};
\draw[thick, red, rounded corners=15pt] (-2.15, 1.05)-- (-1.856,2.233) -- (-0.7,2.28) ;
\draw[thick, red, dashed, rounded corners=15pt] (-2.03, 1)-- (-0.96,1.16) -- (-0.6,2.18) ;
\node at (-1.9,2.4) {$c_1$};
\draw[thick, red, rounded corners=15pt] (2.15, 1.05)-- (1.856,2.233) -- (0.7,2.28) ;
\draw[thick, red, dashed, rounded corners=15pt] (2.03, 1)-- (0.96,1.16) -- (0.6,2.18) ;
\node at (1.9,2.4) {$c_2$};
\draw[thick, red, rotate=80, rounded corners=10pt] (-2.1, 2)-- (-1.5,2.433) -- (-0.7,2.28) ;
\draw[thick, red, dashed, rotate=80, rounded corners=8pt] (-1.3, 1)-- (-0.9,1.36) -- (-0.6,2.18) ;
\node at (-3,-1.4) {$c_g$};
\draw[thick, red, rotate=-80, rounded corners=10pt] (2.1, 2)-- (1.5,2.433) -- (0.7,2.28) ;
\draw[thick, red, dashed, rotate=-80, rounded corners=8pt] (1.3, 1)-- (0.9,1.36) -- (0.6,2.18) ;
\node at (3,-1.4) {$c_3$};
\draw[->, rounded corners=20pt, rotate=-19.5] (0, 6.2)..controls (1.395, 6.039) and (2.716, 5.574).. (3.906,4.817);
\node[rotate=-40] at (4.28,5.284) {$R$};
\draw[->] (-0.6,6.9)..controls (-0.8,6.2) and (0.8,6.2).. (0.6,6.9);
\draw (0,-7) -- (0,-5.6); \draw (0,2) -- (0,3.4); \draw (0,5.6) -- (0,6.25); \draw (0,6.5) -- (0,7.5);
\draw[dotted] (0,-5.3) -- (0,1.85); \draw[dotted] (0,3.7) -- (0,5.3); 
\draw[->, rotate=40] (-0.6,6.9)..controls (-0.8,6.2) and (0.8,6.2).. (0.6,6.9);
\draw[rotate=40]  (0,-7) -- (0,-5.6);  \draw[rotate=40] (0,5.6) -- (0,6.25); \draw[rotate=40] (0,6.5) -- (0,7.5);
\draw[dotted, rotate=40] (0,-5.3) -- (0,5.3); 
\node at (1.4,6.7) {$\rho_1$};
\node at (-5.5,4.8) {$\rho_2$};
\end{scope}
\begin{scope} [xshift=14cm]
 \draw[very thick, violet] (0,0) circle [radius=5.5cm];
 \draw[very thick, violet] (0,2.7) circle [radius=0.8cm]; 
 \draw[very thick, violet, rotate=80] (0,2.7) circle [radius=0.8cm]; 
 \draw[very thick, violet, rotate=-80] (0,2.7) circle [radius=0.8cm]; 
  \draw[very thick, violet, fill ] (0,-2.7) circle [radius=0.03cm]; 
 \draw[very thick, violet, fill, rotate=15] (0,-2.7) circle [radius=0.03cm]; 
\draw[very thick, violet, fill, rotate=-15] (0,-2.7) circle [radius=0.03cm]; 
\draw[thick, blue, rounded corners=5pt] (-2.5, 1.3)--(-2.6, 1.7)..controls (-1.8, 5.3) and (1.8,5.3)..(2.6,1.7)-- (2.5,1.3) ;
\draw[thick, blue, dashed, rounded corners=5pt] (-2.4, 1.3)--(-2.2, 1.7)..controls (-1.8, 4.9) and (1.8,4.9)..(2.2,1.7)-- (2.4,1.3) ;
\draw[thick, red, dashed, rounded corners=15pt] (-2.03, 1)-- (-0.96,1.16) -- (-0.6,2.18) ;
\draw[thick, red, rounded corners=10pt] (-2.1, 1.05)--(-1.8, 1.7)..controls (-4.5,5.6) and (2.3,6.2).. (4.4, 2.3)--(5.15,1.9) ;
\draw[thick, red, rounded corners=6pt] (-0.62, 2.2)--(-1.3,2.5)..controls (-4.1,5.4) and (3,5.7).. (3.5,2)--(3.2,1) ;
\draw[thick, red, dashed, rounded corners=10pt] (3.3, 1)..controls (3.6,0.9) and (4.7,1).. (5.2,1.8) ;
\node at (4.3,3-1) {$d_1$};
\node at (1.5,2.1) {$d_2$};
\node at (0.7, 6.7) {$z$};
\node at (6.7, 0.7) {$y$};
\draw[->, thick] (0,5.8)--(0, 7.3);
\draw[->, thick] (5.8,0)--(7.3,0);
\end{scope}
\end{tikzpicture}
       \caption{The curves $a_i,b_i,c_i,d_i$, the rotation $R$
         and the involutions $\rho_1$ and $\rho_2$ on the surface $\Sigma_g$.}
  \label{fig1}
\end{figure}
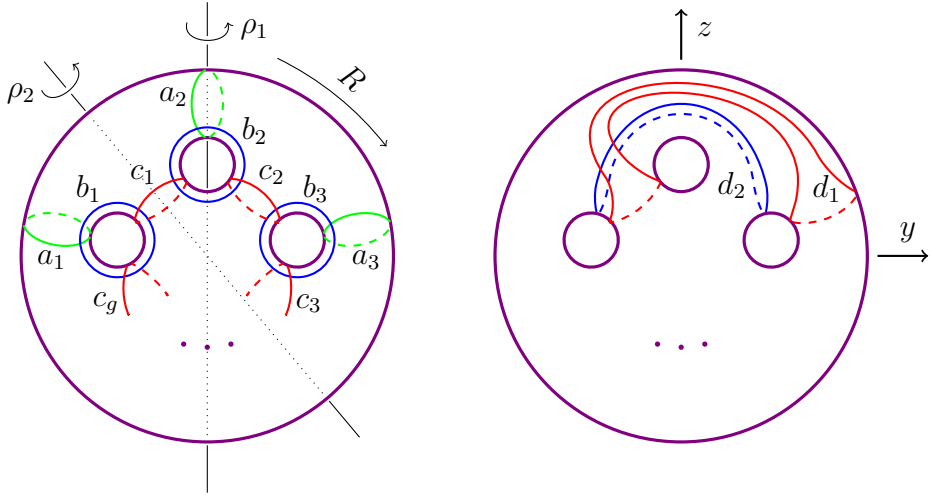

\bigskip


\section{Three new generating sets for $\mod(\Sigma_g)$.}
In this section, we obtain three new generating sets for the mapping class group. The  first one
 will be the first generating set in Corollary~\ref{cor:Dehn-Lick} where
each Dehn  twist is replaced by a product of two opposite Dehn twists.
The new generating sets will allow us to 
generate the mapping class group by a small number of involutions.

Recall that  $R$ denotes the $2\pi/g$--rotation of $\Sigma_g$ represented in
Figure~\ref{fig1}. It is 
 a torsion element of order $g$ in the group $\mod(\Sigma_g)$.  
 Our first generating set is given in the next theorem. For the proof of it,
 following the idea of Wajnryb
 in~\cite{Wajnryb1996},
 we employ the lantern relation. Other generating sets are given as two corollaries
 to the theorem.
 
\begin{theorem} \label{thm:thm4} 
If $g\geq 3$, then the mapping class group $\mod(\Sigma_g)$ is generated by the four elements
\(
R, A_1A_2^{-1},B_1B_2^{-1}, C_1C_2^{-1}. \)
\end{theorem}

\begin{proof}
Let $G$ denote the subgroup of $\mod(\Sigma_g)$
generated by the set 
\[
\{ R, A_1A_2^{-1},B_1B_2^{-1}, C_1C_2^{-1}\}.
\]
 Let $\S$ denote the set of isotopy classes of nonseparating simple 
 closed curves on the surface $\Sigma_g$.
We define a subset $\G$ of $\S\times \S$ as 
\[
\G =\{ (a,b) \ : \   \ A B^{-1}\in G \}.
\]
It is clear that \\
\begin{itemize}
\item (symmetry) if $(a,b)\in \G$ then $(b,a)\in \G$, 
\item (transitivity) if $(a,b)$ and $(b,c)$ are in $\G$ then so is $(a,c)$, and 
\item ($G$--invariance) if $(a,b)\in \G$ and $F\in G$, then $(F(a),F(b))\in \G$.
\end{itemize}
Thus $\G$ is an equivalence relation on $\S$ invariant under the action of $G$. The last property follows from $Ft_at_b^{-1}F^{-1}=t_{F(a)} t_{F(b)}^{-1}$.

Notice that by the very definition of $\G$, the set $\G$ contains 
 the three pairs  $(a_1,a_2), (b_1,b_2)$ and $(c_1,c_2)$.
Since 
\[
R^{k-1}(\alpha_1,\alpha_2)=(\alpha_k,\alpha_{k+1})
\]
 for every $\alpha\in \{a,b,c\}$, the pairs $(a_k,a_{k+1}), (b_k,b_{k+1})$ and $(c_k,c_{k+1})$ are contained in $\G$. 
 It follows from the transitivity that the pairs 
$(a_i,a_{j}), (b_i,b_{j})$ and $ (c_i,c_{j})$ are also contained in $\G$ for all $i,j$. (Here, by abusing the notation
we write $f(a,b)$ for $(f(a),f(b))$, and  
all indices are integers modulo $g$.)

Since
\[
A_1A_2^{-1}B_1B_2^{-1}(a_1,a_3)=(b_1,a_3)
\]
and
\[
B_1B_2^{-1}C_1C_2^{-1}(b_1,a_3)=(c_1,a_3),
\]
we conclude that the pairs $(a_i,b_j), (a_i,c_j), (b_i,c_j)$ are all contained in  the set $\G$ as well. 

As a result of this, we get that if $X,Y\in \{ A_i,B_i,C_i\}$, then 
the mapping class 
$XY^{-1}$ is contained in the subgroup $G$. 

We now use the lantern relation to conclude that $G$ contains one, and hence all, 
of the Dehn twist generators given in Dehn-Lickorish theorem, 
Theorem~\ref{thm:Dehn-Lick}.

It is easy to check that the diffeomorphism 
\[
(B_2A_1^{-1})(C_1A_1^{-1})(A_1A_2^{-1})(C_2A_1^{-1}) 
\]
maps $(b_2,a_1)$ to $(d_1,a_1)$, and the diffeomorphism
\[
(B_3A_1^{-1})(C_2A_1^{-1})(A_3A_1^{-1})(B_3A_1^{-1}) 
\]
maps $(d_1,a_1)$ to $(d_2,a_1)$. Since both of these two 
diffeomorphisms are contained in the subgroup $G$, it follows now from 
the transitivity and the $G$-invariance of $\G$ that  $D_1A_1^{-1}$ and $D_2C_1^{-1}$ are 
contained in $G$. 

Note that $a_1,c_1,c_2,a_3$ bound a sphere with four boundary components.
The Dehn twists about these four curves and about the curves $a_2,d_1,d_2$ 
given in Figure~\ref{fig1}
satisfy the lantern relation 
\[
A_1C_1C_2A_3=A_2D_1D_2,
\]
which may be rewritten as 
\[
A_3=(A_2C_2^{-1})(D_1A_1^{-1})(D_2C_1^{-1}).
\]
Since each factor on the right-hand side is contained in $G$, the 
Dehn twist $A_3$ is also contained in $G$. 
Now from the fact that $A_iA_3^{-1}, B_iA_3^{-1},C_iA_3^{-1}$ are in $G$, 
we conclude that $G$ contains all generators $A_i,B_i$ and $C_i$ of $\mod(\Sigma_g)$
given in Theorem~\ref{thm:Dehn-Lick}.
Consequently, $G=\mod(\Sigma_g)$.

This concludes the proof of the theorem.
\end{proof}

\begin{corollary} \label{cor:gen1} 
If $g\geq 3$, then the mapping class group  $\mod(\Sigma_g)$ is generated by 
the three elements
$R, A_1A_2^{-1}, A_1B_1C_1C_2^{-1}B_3^{-1}A_3^{-1}.$
\end{corollary}

\begin{proof}
Let us denote by
$H$ the subgroup of $\mod(\Sigma_g)$ generated by the set
\[
\{ R, A_1A_2^{-1}, A_1B_1C_1C_2^{-1}B_3^{-1}A_3^{-1}\}.\]
It suffices to prove that $H$ contains $B_1B_2^{-1}$ and
$C_1C_2^{-1}$. 

 It follows from
 \begin{itemize}
   \item $R(a_1,a_2)=(a_2,a_3)$,
   \item $A_1B_1C_1C_2^{-1}B_3^{-1}A_3^{-1}(a_1,a_2)=(b_1,a_2)$,
   \item $A_1B_1C_1C_2^{-1}B_3^{-1}A_3^{-1}(b_1,a_2)=(c_1,a_2)$,
   \item $R(b_1,a_2)=(b_2,a_3)$ and
   \item $R(c_1,a_2)=(c_2,a_3)$
\end{itemize}
 that the elements
 \begin{itemize}
   \item $A_2A_3^{-1}$,
   \item $B_1A_2^{-1}$,
   \item $C_1A_2^{-1}$.
    \item $B_2A_3^{-1} $ and
   \item $C_2A_3^{-1}$
\end{itemize}
are contained in $H$. Now we have
 \begin{itemize}
   \item[$\bullet$] $B_1B_2^{-1}= B_1A_2^{-1}\cdot A_2A_3^{-1}\cdot A_3B_2^{-1}\in H$ and
    \item[$\bullet$] $C_1C_2^{-1}= C_1A_2^{-1}\cdot A_2A_3^{-1}\cdot A_3C_2^{-1}\in H$.
\end{itemize}
It follows from  Theorem~\ref{thm:thm4} that
$H=\mod(\Sigma_g)$, completing the proof of the corollary.
 \end{proof}

\begin{corollary} \label{cor:gen2} 
If $g\geq 8$, then the mapping class group  $\mod(\Sigma_g)$ is generated by 
the there elements
$\rho_1,\rho_2$ and $B_1A_2C_{3} C_4^{-1}A_6^{-1}B_7^{-1}.$
\end{corollary}

\begin{figure}
\begin{tikzpicture}[scale=0.75]
\begin{scope} [xshift=0cm, yshift=0cm]
 \draw[very thick, violet] (0,0) circle [radius=2.5cm];
 \draw[very thick, violet] (0,1.6) circle [radius=0.2cm]; 
 \draw[very thick, violet, rotate=45] (0,1.6) circle [radius=0.2cm]; 
 \draw[very thick, violet, rotate=90] (0,1.6) circle [radius=0.2cm];  
 \draw[very thick, violet, rotate=135] (0,1.6) circle [radius=0.2cm];  
 \draw[very thick, violet, rotate=180] (0,1.6) circle [radius=0.2cm]; 
 \draw[very thick, violet, rotate=-45] (0,1.6) circle [radius=0.2cm]; 
 \draw[very thick, violet, rotate=-90] (0,1.6) circle [radius=0.2cm]; 
 \draw[very thick, violet, rotate=-135] (0,1.6) circle [radius=0.2cm]; 
 \draw[thick, blue, rotate=-90, rounded corners=8pt] (-0.94,1.17)--(-0.48, 1.27)--(-0.18,1.5);  
 \draw[thick, blue, dashed,  rotate=-90, rounded corners=8pt] (-0.94,1.21)--(-0.6, 1.44)--(-0.2,1.54);  
 \draw[thick, red, rotate=-135, rounded corners=8pt] (-0.94,1.17)--(-0.48, 1.27)--(-0.18,1.5);  
 \draw[thick, red, dashed,  rotate=-135, rounded corners=8pt] (-0.94,1.21)--(-0.6, 1.44)--(-0.2,1.54);  
 \draw[thick, blue, rotate=45] (0,1.6) circle [radius=0.3cm]; 
\draw[thick, red, rotate=135] (0,1.6) circle [radius=0.3cm]; 
 \draw[thick, blue, rotate=0, rounded corners=8pt] (-0.02,1.8)--(-0.1, 2.15)--(-0.02,2.5);  
 \draw[thick, blue, dashed, rotate=0, rounded corners=8pt]  (0.02,1.8)--(0.1, 2.15)--(0.02,2.5); 
 \draw[thick, dashed, red, rotate=-180, rounded corners=8pt] (-0.02,1.8)--(-0.1, 2.15)--(-0.02,2.5);  
 \draw[thick, red, rotate=-180, rounded corners=8pt]  (0.02,1.8)--(0.1, 2.15)--(0.02,2.5); 
  \node[scale=0.6, blue] at (-0.7,0.8) {$+$};
  \node[scale=0.6, blue] at (0.02,1.15) {$+$};
  \node[scale=0.6, blue] at (1.1,0.5) {$+$};
  \node[scale=0.6, red] at (1.1,-0.5) {$-$};
  \node[scale=0.6, red] at (0.02,-1.15) {$-$};
  \node[scale=0.6, red] at (-0.7,-0.8) {$-$};
\end{scope}

\begin{scope} [xshift=5.8cm, yshift=0cm, rotate=-45]
 \draw[very thick, violet] (0,0) circle [radius=2.5cm];
 \draw[very thick, violet] (0,1.6) circle [radius=0.2cm]; 
 \draw[very thick, violet, rotate=45] (0,1.6) circle [radius=0.2cm]; 
 \draw[very thick, violet, rotate=90] (0,1.6) circle [radius=0.2cm];  
 \draw[very thick, violet, rotate=135] (0,1.6) circle [radius=0.2cm];  
 \draw[very thick, violet, rotate=180] (0,1.6) circle [radius=0.2cm]; 
 \draw[very thick, violet, rotate=-45] (0,1.6) circle [radius=0.2cm]; 
 \draw[very thick, violet, rotate=-90] (0,1.6) circle [radius=0.2cm]; 
 \draw[very thick, violet, rotate=-135] (0,1.6) circle [radius=0.2cm]; 
 \draw[thick, blue, rotate=-90, rounded corners=8pt] (-0.94,1.17)--(-0.48, 1.27)--(-0.18,1.5);  
 \draw[thick, blue, dashed,  rotate=-90, rounded corners=8pt] (-0.94,1.21)--(-0.6, 1.44)--(-0.2,1.54);  
 \draw[thick, red, rotate=-135, rounded corners=8pt] (-0.94,1.17)--(-0.48, 1.27)--(-0.18,1.5);  
 \draw[thick, red, dashed,  rotate=-135, rounded corners=8pt] (-0.94,1.21)--(-0.6, 1.44)--(-0.2,1.54);  
 \draw[thick, blue, rotate=45] (0,1.6) circle [radius=0.3cm]; 
\draw[thick, red, rotate=135] (0,1.6) circle [radius=0.3cm]; 
 \draw[thick, blue, rotate=0, rounded corners=8pt] (-0.02,1.8)--(-0.1, 2.15)--(-0.02,2.5);  
 \draw[thick, blue, dashed, rotate=0, rounded corners=8pt]  (0.02,1.8)--(0.1, 2.15)--(0.02,2.5); 
 \draw[thick, dashed, red, rotate=-180, rounded corners=8pt] (-0.02,1.8)--(-0.1, 2.15)--(-0.02,2.5);  
 \draw[thick, red, rotate=-180, rounded corners=8pt]  (0.02,1.8)--(0.1, 2.15)--(0.02,2.5); 
  \node[scale=0.6, blue] at (-0.7,0.8) {$+$};
  \node[scale=0.6, blue] at (0.02,1.15) {$+$};
  \node[scale=0.6, blue] at (1.1,0.5) {$+$};
  \node[scale=0.6, red] at (1.1,-0.5) {$-$};
  \node[scale=0.6, red] at (0.02,-1.15) {$-$};
  \node[scale=0.6, red] at (-0.7,-0.8) {$-$};
\end{scope}

\begin{scope} [xshift=11.6cm, yshift=0cm, rotate=-45]
 \draw[very thick, violet] (0,0) circle [radius=2.5cm];
 \draw[very thick, violet] (0,1.6) circle [radius=0.2cm]; 
 \draw[very thick, violet, rotate=45] (0,1.6) circle [radius=0.2cm]; 
 \draw[very thick, violet, rotate=90] (0,1.6) circle [radius=0.2cm];  
 \draw[very thick, violet, rotate=135] (0,1.6) circle [radius=0.2cm];  
 \draw[very thick, violet, rotate=180] (0,1.6) circle [radius=0.2cm]; 
 \draw[very thick, violet, rotate=-45] (0,1.6) circle [radius=0.2cm]; 
 \draw[very thick, violet, rotate=-90] (0,1.6) circle [radius=0.2cm]; 
 \draw[very thick, violet, rotate=-135] (0,1.6) circle [radius=0.2cm]; 
 \draw[thick, blue, rotate=-90, rounded corners=8pt] (-0.94,1.17)--(-0.48, 1.27)--(-0.18,1.5);  
 \draw[thick, blue, dashed,  rotate=-90, rounded corners=8pt] (-0.94,1.21)--(-0.6, 1.44)--(-0.2,1.54);  
 \draw[thick, red, rotate=-135, rounded corners=8pt] (-0.94,1.17)--(-0.48, 1.27)--(-0.18,1.5);  
 \draw[thick, red, dashed,  rotate=-135, rounded corners=8pt] (-0.94,1.21)--(-0.6, 1.44)--(-0.2,1.54);  
 \draw[thick, red, rotate=180] (0,1.6) circle [radius=0.3cm]; 
\draw[thick, red, rotate=135] (0,1.6) circle [radius=0.3cm]; 
 \draw[thick, blue, rotate=0, rounded corners=8pt] (-0.02,1.8)--(-0.1, 2.15)--(-0.02,2.5);  
 \draw[thick, blue, dashed, rotate=0, rounded corners=8pt]  (0.02,1.8)--(0.1, 2.15)--(0.02,2.5); 
 \draw[thick, dashed, blue, rotate=45, rounded corners=8pt] (-0.02,1.8)--(-0.1, 2.15)--(-0.02,2.5);  
 \draw[thick, blue, rotate=45, rounded corners=8pt]  (0.02,1.8)--(0.1, 2.15)--(0.02,2.5); 
  \node[scale=0.6, blue] at (-0.7,0.8) {$+$};
  \node[scale=0.6, blue] at (0.02,1.15) {$+$};
  \node[scale=0.6, blue] at (1.1,0.5) {$+$};
  \node[scale=0.6, red] at (1.1,-0.5) {$-$};
  \node[scale=0.6, red] at (0.02,-1.15) {$-$};
  \node[scale=0.6, red] at (-0.7,-0.8) {$-$};
\end{scope}

\begin{scope} [xshift=8.7cm, yshift=-5.6cm, rotate=0]
 \draw[very thick, violet] (0,0) circle [radius=2.5cm];
 \draw[very thick, violet] (0,1.6) circle [radius=0.2cm]; 
 \draw[very thick, violet, rotate=45] (0,1.6) circle [radius=0.2cm]; 
 \draw[very thick, violet, rotate=90] (0,1.6) circle [radius=0.2cm];  
 \draw[very thick, violet, rotate=135] (0,1.6) circle [radius=0.2cm];  
 \draw[very thick, violet, rotate=180] (0,1.6) circle [radius=0.2cm]; 
 \draw[very thick, violet, rotate=-45] (0,1.6) circle [radius=0.2cm]; 
 \draw[very thick, violet, rotate=-90] (0,1.6) circle [radius=0.2cm]; 
 \draw[very thick, violet, rotate=-135] (0,1.6) circle [radius=0.2cm]; 
 \draw[thick, blue, rotate=-90, rounded corners=8pt] (-0.94,1.17)--(-0.48, 1.27)--(-0.18,1.5);  
 \draw[thick, blue, dashed,  rotate=-90, rounded corners=8pt] (-0.94,1.21)--(-0.6, 1.44)--(-0.2,1.54);  
 \draw[thick, red, rotate=-135, rounded corners=8pt] (-0.94,1.17)--(-0.48, 1.27)--(-0.18,1.5);  
 \draw[thick, red, dashed,  rotate=-135, rounded corners=8pt] (-0.94,1.21)--(-0.6, 1.44)--(-0.2,1.54);  
 \draw[thick, red, rotate=180] (0,1.6) circle [radius=0.3cm]; 
\draw[thick, red, rotate=135] (0,1.6) circle [radius=0.3cm]; 
 \draw[thick, blue, rotate=0, rounded corners=8pt] (-0.02,1.8)--(-0.1, 2.15)--(-0.02,2.5);  
 \draw[thick, blue, dashed, rotate=0, rounded corners=8pt]  (0.02,1.8)--(0.1, 2.15)--(0.02,2.5); 
 \draw[thick, dashed, blue, rotate=45, rounded corners=8pt] (-0.02,1.8)--(-0.1, 2.15)--(-0.02,2.5);  
 \draw[thick, blue, rotate=45, rounded corners=8pt]  (0.02,1.8)--(0.1, 2.15)--(0.02,2.5); 
  \node[scale=0.6, blue] at (-0.7,0.8) {$+$};
  \node[scale=0.6, blue] at (0.02,1.15) {$+$};
  \node[scale=0.6, blue] at (1.1,0.5) {$+$};
  \node[scale=0.6, red] at (1.1,-0.5) {$-$};
  \node[scale=0.6, red] at (0.02,-1.15) {$-$};
  \node[scale=0.6, red] at (-0.7,-0.8) {$-$};
\end{scope}

\begin{scope} [xshift=2.9cm, yshift=-5.6cm, rotate=0]
 \draw[very thick, violet] (0,0) circle [radius=2.5cm];
 \draw[very thick, violet] (0,1.6) circle [radius=0.2cm]; 
 \draw[very thick, violet, rotate=45] (0,1.6) circle [radius=0.2cm]; 
 \draw[very thick, violet, rotate=90] (0,1.6) circle [radius=0.2cm];  
 \draw[very thick, violet, rotate=135] (0,1.6) circle [radius=0.2cm];  
 \draw[very thick, violet, rotate=180] (0,1.6) circle [radius=0.2cm]; 
 \draw[very thick, violet, rotate=-45] (0,1.6) circle [radius=0.2cm]; 
 \draw[very thick, violet, rotate=-90] (0,1.6) circle [radius=0.2cm]; 
 \draw[very thick, violet, rotate=-135] (0,1.6) circle [radius=0.2cm]; 
 \draw[thick, blue, rotate=-90, rounded corners=8pt] (-0.94,1.17)--(-0.48, 1.27)--(-0.18,1.5);  
 \draw[thick, blue, dashed,  rotate=-90, rounded corners=8pt] (-0.94,1.21)--(-0.6, 1.44)--(-0.2,1.54);  
 \draw[thick, red, rotate=-135, rounded corners=8pt] (-0.94,1.17)--(-0.48, 1.27)--(-0.18,1.5);  
 \draw[thick, red, dashed,  rotate=-135, rounded corners=8pt] (-0.94,1.21)--(-0.6, 1.44)--(-0.2,1.54);  
\draw[thick, red, rotate=-180, rounded corners=8pt] (-0.94,1.17)--(-0.48, 1.27)--(-0.18,1.5);  
 \draw[thick, red, dashed,  rotate=-180, rounded corners=8pt] (-0.94,1.21)--(-0.6, 1.44)--(-0.2,1.54);  
\draw[thick, red, rotate=135] (0,1.6) circle [radius=0.3cm]; 
 \draw[thick, blue, rotate=0, rounded corners=8pt] (-0.02,1.8)--(-0.1, 2.15)--(-0.02,2.5);  
 \draw[thick, blue, dashed, rotate=0, rounded corners=8pt]  (0.02,1.8)--(0.1, 2.15)--(0.02,2.5); 
 \draw[thick, dashed, blue, rotate=45, rounded corners=8pt] (-0.02,1.8)--(-0.1, 2.15)--(-0.02,2.5);  
 \draw[thick, blue, rotate=45, rounded corners=8pt]  (0.02,1.8)--(0.1, 2.15)--(0.02,2.5); 
  \node[scale=0.6, blue] at (-0.7,0.8) {$+$};
  \node[scale=0.6, blue] at (0.02,1.15) {$+$};
  \node[scale=0.6, blue] at (1.1,0.5) {$+$};
  \node[scale=0.6, red] at (1.1,-0.5) {$-$};
  \node[scale=0.6, red] at (0.45,-1.11) {$-$};
  \node[scale=0.6, red] at (-0.7,-0.8) {$-$};
\end{scope}

\begin{scope} [xshift=0cm, yshift=-11.2cm, rotate=0]
 \draw[very thick, violet] (0,0) circle [radius=2.5cm];
 \draw[very thick, violet] (0,1.6) circle [radius=0.2cm]; 
 \draw[very thick, violet, rotate=45] (0,1.6) circle [radius=0.2cm]; 
 \draw[very thick, violet, rotate=90] (0,1.6) circle [radius=0.2cm];  
 \draw[very thick, violet, rotate=135] (0,1.6) circle [radius=0.2cm];  
 \draw[very thick, violet, rotate=180] (0,1.6) circle [radius=0.2cm]; 
 \draw[very thick, violet, rotate=-45] (0,1.6) circle [radius=0.2cm]; 
 \draw[very thick, violet, rotate=-90] (0,1.6) circle [radius=0.2cm]; 
 \draw[very thick, violet, rotate=-135] (0,1.6) circle [radius=0.2cm]; 
 \draw[thick, red, rotate=135, rounded corners=8pt] (-0.94,1.17)--(-0.48, 1.27)--(-0.18,1.5);  
 \draw[thick, red, dashed,  rotate=135, rounded corners=8pt] (-0.94,1.21)--(-0.6, 1.44)--(-0.2,1.54);  
\draw[thick, blue, rotate=45] (0,1.6) circle [radius=0.3cm]; 
 \draw[thick, blue, rotate=0, rounded corners=8pt] (-0.02,1.8)--(-0.1, 2.15)--(-0.02,2.5);  
 \draw[thick, blue, dashed, rotate=0, rounded corners=8pt]  (0.02,1.8)--(0.1, 2.15)--(0.02,2.5); 
\draw[thick, red, dashed, rotate=180, rounded corners=8pt] (-0.02,1.8)--(-0.1, 2.15)--(-0.02,2.5);  
 \draw[thick, red, rotate=180, rounded corners=8pt]  (0.02,1.8)--(0.1, 2.15)--(0.02,2.5); 
  \node[scale=0.6, blue] at (-0.7,0.8) {$+$};
  \node[scale=0.6, blue] at (0.02,1.15) {$+$};
  \node[scale=0.6, red] at (0,-1.2) {$-$};
  \node[scale=0.6, red] at (-0.4,-1) {$-$};
\end{scope}

\begin{scope} [xshift=5.8cm, yshift=-11.2cm, rotate=-45]
 \draw[very thick, violet] (0,0) circle [radius=2.5cm];
 \draw[very thick, violet] (0,1.6) circle [radius=0.2cm]; 
 \draw[very thick, violet, rotate=45] (0,1.6) circle [radius=0.2cm]; 
 \draw[very thick, violet, rotate=90] (0,1.6) circle [radius=0.2cm];  
 \draw[very thick, violet, rotate=135] (0,1.6) circle [radius=0.2cm];  
 \draw[very thick, violet, rotate=180] (0,1.6) circle [radius=0.2cm]; 
 \draw[very thick, violet, rotate=-45] (0,1.6) circle [radius=0.2cm]; 
 \draw[very thick, violet, rotate=-90] (0,1.6) circle [radius=0.2cm]; 
 \draw[very thick, violet, rotate=-135] (0,1.6) circle [radius=0.2cm]; 
 \draw[thick, red, rotate=135, rounded corners=8pt] (-0.94,1.17)--(-0.48, 1.27)--(-0.18,1.5);  
 \draw[thick, red, dashed,  rotate=135, rounded corners=8pt] (-0.94,1.21)--(-0.6, 1.44)--(-0.2,1.54);  
\draw[thick, blue, rotate=45] (0,1.6) circle [radius=0.3cm]; 
 \draw[thick, blue, rotate=0, rounded corners=8pt] (-0.02,1.8)--(-0.1, 2.15)--(-0.02,2.5);  
 \draw[thick, blue, dashed, rotate=0, rounded corners=8pt]  (0.02,1.8)--(0.1, 2.15)--(0.02,2.5); 
\draw[thick, red, dashed, rotate=180, rounded corners=8pt] (-0.02,1.8)--(-0.1, 2.15)--(-0.02,2.5);  
 \draw[thick, red, rotate=180, rounded corners=8pt]  (0.02,1.8)--(0.1, 2.15)--(0.02,2.5); 
  \node[scale=0.6, blue] at (-0.7,0.8) {$+$};
  \node[scale=0.6, blue] at (0.02,1.15) {$+$};
  \node[scale=0.6, red] at (0,-1.2) {$-$};
  \node[scale=0.6, red] at (-0.4,-1) {$-$};
\end{scope}

\begin{scope} [xshift=11.6cm, yshift=-11.2cm, rotate=-45]
 \draw[very thick, violet] (0,0) circle [radius=2.5cm];
 \draw[very thick, violet] (0,1.6) circle [radius=0.2cm]; 
 \draw[very thick, violet, rotate=45] (0,1.6) circle [radius=0.2cm]; 
 \draw[very thick, violet, rotate=90] (0,1.6) circle [radius=0.2cm];  
 \draw[very thick, violet, rotate=135] (0,1.6) circle [radius=0.2cm];  
 \draw[very thick, violet, rotate=180] (0,1.6) circle [radius=0.2cm]; 
 \draw[very thick, violet, rotate=-45] (0,1.6) circle [radius=0.2cm]; 
 \draw[very thick, violet, rotate=-90] (0,1.6) circle [radius=0.2cm]; 
 \draw[very thick, violet, rotate=-135] (0,1.6) circle [radius=0.2cm]; 
 \draw[thick, red, rotate=135, rounded corners=8pt] (-0.94,1.17)--(-0.48, 1.27)--(-0.18,1.5);  
 \draw[thick, red, dashed,  rotate=135, rounded corners=8pt] (-0.94,1.21)--(-0.6, 1.44)--(-0.2,1.54);  
 
 \draw[thick, blue, rotate=45, rounded corners=8pt] (-0.02,1.8)--(-0.1, 2.15)--(-0.02,2.5);  
 \draw[thick, blue, dashed, rotate=45, rounded corners=8pt]  (0.02,1.8)--(0.1, 2.15)--(0.02,2.5); 

 \draw[thick, blue, rotate=0, rounded corners=8pt] (-0.02,1.8)--(-0.1, 2.15)--(-0.02,2.5);  
 \draw[thick, blue, dashed, rotate=0, rounded corners=8pt]  (0.02,1.8)--(0.1, 2.15)--(0.02,2.5); 

\draw[thick, red, dashed, rotate=180, rounded corners=8pt] (-0.02,1.8)--(-0.1, 2.15)--(-0.02,2.5);  
 \draw[thick, red, rotate=180, rounded corners=8pt]  (0.02,1.8)--(0.1, 2.15)--(0.02,2.5); 
   \node[scale=0.6, blue] at (-0.7,0.8) {$+$};
  \node[scale=0.6, blue] at (0.02,1.15) {$+$};
  \node[scale=0.6, red] at (0,-1.2) {$-$};
  \node[scale=0.6, red] at (-0.4,-1) {$-$};
\end{scope}
	 \node[scale=0.8] at (-1.5,2.6) {$F_1$};
 	\node[scale=0.8] at (4.3,2.6) {$F_2$};
 	\node[scale=0.8] at (10.1,2.6) {$F_3$};
 	\node[scale=0.8] at (7.2,-3.1) {$F_4$};
 	\node[scale=0.8] at (1.4,-3.1) {$F_5$};
 	\node[scale=0.8] at (-1.5,-8.7) {$F_6$};
 	\node[scale=0.8] at (4.3,-8.7) {$F_7$};
 	\node[scale=0.8] at (10.1,-8.7) {$F_8$};
 	\draw[ ->, rounded corners=5pt] (2.4,1.5)--(2.9,1.6)-- (3.4,1.5);
 	\node[scale=0.8] at (2.9,2) {$R$};
 	\draw[ ->, rounded corners=5pt] (8.2,1.5)--(8.7,1.6)-- (9.2,1.5);
 	\node[scale=0.8] at (8.7,2) {$F_2F_1$};
 	\draw[ ->, rounded corners=5pt] (11.1,-2.7)--(11, -3.3)--(10.7,-3.7);
 	\node[scale=0.8] at (11.6,-3.2) {$R^{-1}$};
 	\draw[ <-, rounded corners=5pt] (5.3,-4.3)--(5.8,-4.2)-- (6.3,-4.3);
 	\node[scale=0.8] at (5.8,-3.8) {$F_4F_3$};
 	\draw[ ->, rounded corners=5pt] (2.4,-9.8)--(2.9,-9.7)-- (3.4,-9.8);
 	\node[scale=0.8] at (2.9,-9.4) {$R$};
 	\draw[ ->, rounded corners=5pt] (8.2,-9.8)--(8.7,-9.7)-- (9.2,-9.8);
 	\node[scale=0.8] at (8.7,-9.4) {$F_7F_6$};
\end{tikzpicture}
 	\caption{Proof of Corollary~\ref{cor:gen2} for $g=8$.}
\end{figure}
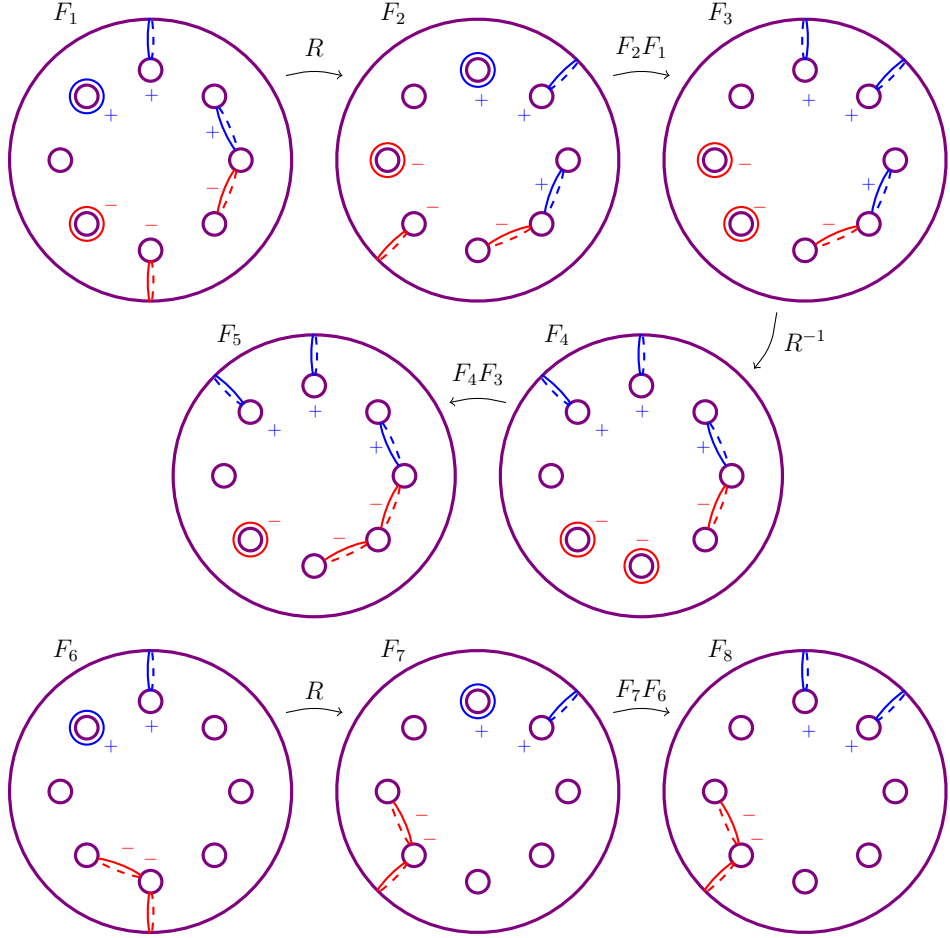

\begin{proof}
Let $F_1= B_1A_2C_{3} C_4^{-1}A_6^{-1}B_7^{-1}$ and 
let $N$ denote the subgroup of $\mod(\Sigma_g)$ generated by the set
 \[
 \{ \rho_1,\rho_2,F_1\}.\]
 Then the rotation $R=\rho_1\rho_2$ is in $N$.
By Theorem~\ref{thm:thm4}, it suffices to prove that $N$ contains $A_1A_2^{-1}$, $B_1B_2^{-1}$ and $C_1C_2^{-1}$. 

Let $F_2$ denote the conjugation of $F_1$ by $R$:
\[
F_2 = RF_1R^{-1}=B_2A_3C_{4} C_5^{-1}A_7^{-1}B_8^{-1}.
\]
It is easy to show that
\[
F_2F_1(b_{2},a_{3},c_4,c_5,a_7,b_8)= (a_{2},a_{3},c_4,c_5,b_7,b_8),
\]
so that 
$F_3= A_{2}A_{3}C_4C_5^{-1}B_7^{-1}B_8^{-1}$ is contained in $N$. 

Now let 
\[ 
F_4=R^{-1}F_3R=A_{1}A_{2}C_3C_4^{-1}B_6^{-1}B_7^{-1}.
\]
It can also be shown that
\[ 
F_4F_3 (a_{1},a_{2},c_3,c_4,b_6,b_7)=(a_{1},a_{2},c_3,c_4,c_5,b_7).
\]
Hence, the subgroup $N$ contains the element 
\[
F_5=A_{1}A_{2}C_3C_4^{-1}C_5^{-1}B_7^{-1}.
\]

Therefore, 
\[
F_4^{-1}F_5=B_6C_5^{-1}\in N.
\]
By conjugating this element by powers of $R$ we see that $N$ contains the elements 
$B_{i+1}C_i^{-1}$. In particular, $B_2C_1^{-1}\in N$, and hence  
$\rho_2(B_2C_1^{-1})\rho_2=B_1C_1^{-1}\in N$. It follows that
$B_iC_i^{-1}\in N$ for all $i$. Now, we have
\begin{itemize}
   \item[$\bullet$] $B_1B_2^{-1}= (B_1C_1^{-1}) (C_1B_2^{-1})\in N$, and
    \item[$\bullet$] $C_1C_2^{-1}= (C_1B_2^{-1}) (B_2C_2^{-1})\in N$.
\end{itemize}
It remains to show that $A_1A_2^{-1}\in N$.

Let 
\begin{eqnarray*}
F_6
&=&F_1(C_3^{-1}C_4)(B_7C_6^{-1})\\
&=& B_1A_2A_6^{-1} C_6^{-1}
\end{eqnarray*}
and let
\begin{eqnarray*}
F_7
&=&RF_6R^{-1}\\
&=& B_2A_3A_7^{-1} C_7^{-1}.
\end{eqnarray*}
It can be verified that
\[
F_7F_6(b_2,a_3,a_7,c_7)=(a_2,a_3,a_7,c_7)
\]
so that $N$ contains the element
\[
F_8= A_2A_3A_7^{-1} C_7^{-1}.
\]
Thus we get that
\[
F_8F_7^{-1}=A_2B_2^{-1}
\]
is in $N$. By conjugating this with powers of $R$ we see that
$A_iB_i^{-1}$ is in $N$ for all $i$.
Therefore
\begin{itemize}
\item[$\bullet$] $A_1A_2^{-1}=(A_1B_1^{-1})(B_1B_2^{-1})(B_2A_2^{-1})$
\end{itemize}
is contained in $N$. Consequently,  $N=\mod(\Sigma_g)$. 

This  completes the proof of the corollary.
 \end{proof}

\bigskip

\section{Involution generators}
Recall that an involution in a group $G$ is an element of order $2$.
If $\rho$ is an involution in $G$ conjugating $x$ to $y$,  then
\[
(\rho xy^{-1})^2=\rho xy^{-1}\rho xy^{-1}=y x^{-1} xy^{-1}=1.
\]
 Thus we have the following elementary but useful lemma.

\begin{lemma} \label{lem:invol}
If $\rho$ is an involution in a group $G$ and if $x$ and $y$ are 
elements in $G$ satisfying $\rho x \rho=y$, then $\rho xy^{-1}$ is an involution.
\end{lemma}

Consider now the surface $\Sigma_g$ of genus $g\geq 3$.
Since $\rho_2(a_1)=a_2$,  $\rho_1(a_1)=a_3$, $\rho_1(b_1)=b_3$ and
$\rho_1(c_1)=c_2$, we have
\[
\rho_2 A_1\rho_2=A_2 \ \mbox{ and }\ \rho_1 A_1B_1C_1 \rho_1= A_3B_3C_2.
\]
 It follows now from Lemma~\ref{lem:invol} that 
\[
\rho_2 A_1A_2^{-1} \ \mbox{ and } \  \rho_1 A_1B_1C_1 C_2^{-1}B_3^{-1}A_3^{-1}
\]
are involutions. 

Let $\rho_3=R^2\rho_1R^{-2}$.  For $g\geq 8$, we have
\[
\rho_3 B_{1}A_2C_3\rho_3= B_7A_6C_4,
\]
so that
$\rho_3 B_{1}A_2C_3 C_4^{-1} A_6^{-1} B_7^{-1} $ is an involution.

Finally, we state our main result.

\begin{theorem} \label{thm:thm7} 
Let $\Sigma_g$ be the closed connected oriented surface of genus $g$.
Then the mapping class group $\mod(\Sigma_g)$ is generated by the involutions
 \begin{itemize}
   \item[$(1)$]  $\rho_1,\rho_2$ and $\rho_3 B_{1}A_2C_3 C_4^{-1} A_6^{-1} B_7^{-1} $ if $g\geq 8$, and
   \item[$(2)$] 
   $\rho_1, \rho_2, \rho_2A_1A_2^{-1}, \rho_1A_1B_1C_1C_2^{-1}B_3^{-1}A_3^{-1}$
   if  $g\geq 3$.
\end{itemize}
\end{theorem}
\begin{proof}
The proof follows at once from Corollary~\ref{cor:gen2}, Corollary~\ref{cor:gen1} and the fact that $R=\rho_1\rho_2$.
\end{proof}

Since the first homology groups of $\mod(\Sigma_1)$ and $\mod(\Sigma_2)$
are isomorphic to the cyclic group $\Z_{12}$ and $\Z_{10}$ respectively, 
these two mapping class groups cannot be generated by involutions. In fact,
the group $\mod(\Sigma_1)$ is isomorphic to $\SL (2,\Z)$ and 
$ -I$ is the only element of order two in  $\SL (2,\Z)$, where $I$ denotes 
the identity matrix.
It was shown by Stukow~\cite{Stukow2} that the subgroup of $\mod(\Sigma_2)$
generated involutions is of index five. 



%
%
%
%

\end{document}